\documentclass[10pt]{amsart}
\usepackage{latexsym,amsmath,amssymb}

 \renewcommand{\epsilon}{\varepsilon}
 \newcommand{\newsection}[1]
 {\subsection{#1}\setcounter{theorem}{0} \setcounter{equation}{0}
 \par\noindent}

 \newtheorem{theorem}{Theorem}
 
 \newtheorem{lemma}[theorem]{Lemma}
 \newtheorem{corr}[theorem]{Corollary}
 
 \newtheorem{proposition}[theorem]{Proposition}
 \newtheorem{deff}[theorem]{Definition}
 
 \newcommand{\bth}{\begin{theorem}}
 \newcommand{\ble}{\begin{lemma}}
 \newcommand{\bcor}{\begin{corr}}
 \newcommand{\bdeff}{\begin{deff}}
 \newcommand{\bprop}{\begin{proposition}}
 \newcommand{\ele}{\end{lemma}}
 \newcommand{\ecor}{\end{corr}}
 \newcommand{\edeff}{\end{deff}}
 
 \newcommand{\eprop}{\end{proposition}}

 \newcommand{\cd}{\, \cdot\, }

 \newcommand{\Rn}{{\mathbb R}^n}
 \newcommand{\Rf}{{\mathbb R}^4}

 \renewcommand{\Pi}{\varPi}

 \renewcommand{\epsilon}{\varepsilon}

  \newcommand{\R}{{\mathbb R}}

 \newcommand{\Kob}{{\mathcal K}}
 \newcommand{\Kout}{\Rn\backslash \Kob}
 \newcommand{\Kfout}{{\mathbb R}^4\backslash \Kob}

 \newcommand{\Stk}{S_T^\Kob}
 \newcommand{\extn}{{\R^n\backslash\mathcal{K}}}

\title[Global existence for nonlinear Dirichlet-wave equations]
{Concerning the Strauss conjecture and almost global existence for nonlinear Dirichlet-wave 
equations 
in 
$4$-dimensions}
\author{Yi Du}
\author{Jason Metcalfe}
\author{Christopher D. Sogge}
\author{Yi Zhou}
\address{School of Mathematical Science, Fudan University}
\address{Department of Mathematics, University of North Carolina, Chapel Hill}
\address{Department of Mathematics, Johns Hopkins University}
\address{School of Mathematical Science, Fudan University}

\date{October 4, 2007}

\begin{document}
\maketitle

\newsection{Introduction}

The purpose of this paper is to show that certain sharp existence theorems for small amplitude nonlinear wave equations in the Minkowski space setting extend to the case of nonlinear Dirichlet-wave equations outside of obstacles.  Our main result is that the obstacle version of the Strauss conjecture holds when the spatial dimension is equal to 4.  The corresponding result for 4-dimensional Minkowski space was first proved by one of us \cite{Zhou}.  In this paper we shall also show that H\"ormander's \cite{H} almost global existence theorem for quasilinear equations extends to nonlinear Dirichlet-wave equations.  

In recent years there has been much work on semilinear and quasilinear
nonlinear Dirichlet-wave equations with nonlinearities depending on
first and second derivatives of the solution $u$, but not on $u$
itself.  Recently, the first and last authors \cite{YY} were able to show that Lindblad's sharp existence theorem, involving general quadratic nonlinearities of all three types,  $Q(u,u',u'')$,  for $ \R_+\times{\mathbb R}^3$ extends to Dirichlet wave equations in $\R_+\times{\mathbb R}^3\backslash \Kob$ if $\Kob$ is star-shaped.  This was done by coupling Hardy-type estimates with a variant the weighted local energy estimates of Keel, Smith and one of the authors \cite{KSS}, \cite{KSS2}.  The variant involves weighted $L^2(dtdx)$ estimates for solutions $u$ of constant coefficient linear wave equations rather than the more traditional ``KSS estimates" involving the space-time gradient, $u'$.  This key weighted estimate was proved in \cite{YY} by adapting the earlier arguments from \cite{KSS} and \cite{KSS2}, but in this paper, we shall see that the estimates for constant coefficient wave equations are actually are an immediate corollary of the ``KSS estimates".  Estimates of this type are all one needs to handle semilinear equations.  To prove results for quasilinear equations we shall couple these constant coefficient estimates with 
 variable coefficient weighted $L^2$ estimates of the second and third authors \cite{MS}.  This allows a somewhat simpler approach than the one in \cite{YY} since one does not have to use the scaling vector field $L=t\partial_t + \langle \, x,\nabla_x\, \rangle$.  An apparent compromise, though, is that, while we can show that the natural variant of the Strauss conjecture is valid outside of nontrapping obstacles, our almost global existence results for quasilinear equations are for the star-shaped case.  On the other hand, our arguments allow one to give a different proof of H\"ormander's almost global existence theorem for $\R_+\times \R^4$ that avoids somewhat delicate commutator arguments involving the operator
\begin{equation}\label{D}
|D|^{-1}=1/\sqrt{-\Delta}\end{equation}
and the coefficients.
Here $\Delta=\partial^2_1+\partial^2_2+\cdots+\partial^2_n$, of course, is the standard Laplacian in $\Rn$.


Let us now state our main results.

 Our first is for equations of the form
\begin{equation}\label{Str}
\begin{cases}
\square u(t,x)=F_\sigma\bigl(u(t,x)\bigr)), \quad (t,x)\in \R_+\times \Kfout
\\
u(t,x)=0, \quad x\in \partial\Kob
\\
u(0,x)= f(x), \quad \partial_t u(0,x)= g(x), \quad x\in \Kfout,
\end{cases}
\end{equation}
where $\square = \partial_t^2-\Delta$.
We shall here assume that $f,g\in C^\infty(\Kfout)$ are fixed and vanish for large $x$, although this condition can be weakened to one where the data has certain small weighted Sobolev norms that are related to the estimates to follow.  We assume further that the nonlinear term behaves like $|u|^\sigma$, and so we assume that it satisfies the estimates
\begin{equation}\label{ass1}\sum_{0\le j\le 2} |u|^j \bigl|\partial_u^j F_\sigma(u)\bigr|\le C|u|^\sigma,
\end{equation}
when $u$ is small.

In order to solve \eqref{Str}, we must impose compatibility assumptions on the data $(f,g)$.  Such conditions are well known and we refer the reader to \cite{KSS2}.  Briefly,  for nonlinear equations of the form $\square u=N(u,u',u'')$\footnotemark, if we let $J_ku=\{\partial^\alpha_x u: \, 0\le |\alpha|\le k\}$, we can write $\partial_t^ku(0,\cd)=\psi_k(J_kf,J_{k-1}g)$, $0\le k\le m$, where $u$ is any formal $H^m$ local solution to the nonlinear Dirichlet-wave equation and $m$ is fixed.  The $\psi_k$ are called compatibility functions and depend on the nonlinearity $N$, $J_kf$ and $J_{k-1}g$.  The compatibility condition of order $m$ for the Dirichlet-wave equation $\square u=N(u,u',u'')$ with data $(f,g)\in H^m\times H^{m-1}$ requires that the $\psi_k$ vanish on $\partial\Kob$ when $0\le k\le m-1$.  Additionally, we say that $(f,g)\in C^\infty$ satisfy the compatibility condition to infinite order if the above condition holds for all $m$.  We shall assume that this is the case in the results to follow.

\footnotetext{Here, and in what follows, $u'=\partial u=\nabla u = (\partial_t
  u,\nabla_x u)$ denotes space-time gradients unless otherwise noted
  by a subscript.}

We can now state our first main result.

\begin{theorem}\label{theorem1}  Suppose that $F_\sigma$ is as above and that $\Kob\subset \Rf $
is a compact nontrapping obstacle with smooth boundary.  Suppose also that the data $(f,g)$ in \eqref{Str} vanish for $|x|>R$, with $R$ fixed, and satisfy the compatibility condition of order $3$.  Then, if $\sigma>2$, there is an $\varepsilon_0>0$, depending only on $\Kob$ and the constants in \eqref{ass1} so that \eqref{Str} has a global solution satisfying $(u(t,\cd),\partial_tu(t,\cd))\in H^3\times H^2$ for every $t>0$ assuming that
\begin{equation}\label{Stdata}
\sum_{|\alpha|\le 3}\|\partial^\alpha_x f\|_{L^2(\Kfout)}+\sum_{|\alpha|\le 2}\|\partial^\alpha_x g\|_{L^2(\Kfout)}\le \varepsilon,
\end{equation}
with $0<\varepsilon<\varepsilon_0$.
\end{theorem}

For simplicity we are assuming that the data are compactly supported.  The proof of the above result can easily be adapted to show that this assumption can be removed if for some fixed $\delta>0$ we assume that
$$
\sum_{|\alpha|\le 3} \| \langle x\rangle^{1+\delta}\partial^\alpha_x f\|_{L^2(\Rf\backslash \Kob)}+\sum_{|\alpha|\le 2} \| \langle x\rangle^{2+\delta}\partial^\alpha_x g\|_{L^2(\Rf\backslash \Kob)}
$$
is small.

Theorem \ref{theorem1} settles an $n=4$ exterior domain analog of the
Strauss conjecture \cite{Strauss2}, which
 regards global existence to equations of the
form \eqref{Str} for small data when there is no boundary.  In
particular, it states that there are global solutions provided that
$\sigma>p_c$ where $p_c>1$ solves
\[(n-1)p_c^2-(n+1)p_c-2=0.\]
This conjecture was inspired by the earlier work of John \cite{John} which
showed both global existence for $\sigma>p_c=1+\sqrt{2}$ and finite time
blow-up for $\sigma<p_c$ when $n=3$.  In $\R\times\R^n$, the conjecture has
since been resolved in all dimensions by Georgiev, Lindblad, and Sogge
\cite{GLS} and Tataru \cite{Ta}.  The interested reader should also consult the
references therein for some preceding partial results.  In particular,
as we mentioned before, the boundaryless $n=4$ case was first settled
by Zhou \cite{Zhou}.  A result of Sideris \cite{Si} shows that finite
time blow-up can occur for $\sigma<p_c$ for the corresponding
equations in $\R_+\times \R^n$, and in particular, shows that
$\sigma>2$ is necessary in Theorem \ref{theorem1}.


Our other main result concerns quasilinear equations of the form
\begin{equation}\label{q}
\begin{cases}
\square u(t,x)=Q(u,u',u''), \quad (t,x)\in \R_+\times \Kfout
\\
u(t,x)=0, \quad x\in \partial\Kob
\\
u(0,x)= f(x), \quad \partial_t u(0,x)= g(x), \quad x\in \Kfout,
\end{cases}
\end{equation}
where $Q$ is a smooth function of its arguments.
We assume further that the nonlinearity is of the form
\begin{equation}\label{form}
Q(u,u',u'')=A(u,u')+\sum_{0\le j,k\le 4} B^{jk}(u,u')\partial_j\partial_ku.
\end{equation}
Here and in what follows $\partial_0=\partial_t$. 
We assume that 
$A$ vanishes to second order at $(u,u')=(0,0)$, and that $B^{jk}(0,0)=0$ for all $0\le j,k\le 4$.  Additionally, we of course assume the symmetry condition
\begin{equation}\label{sym}
B^{jk}(u,u')=B^{kj}(u,u'), \quad 0\le j,k\le 4.
\end{equation}

Our other main result says that if the data is small enough then there is almost global existence for \eqref{q}:

\begin{theorem}\label{qthm}  Let $\Kob\subset \Rf$ be a smooth compact star-shaped obstacle and let $Q(u,u',u'')$ be as above.  Assume further that $(f,g)\in C^\infty(\Rf\backslash \Kob)$ vanishes for $|x|>R$, with $R$ fixed, and satisfies the compatibility conditions to infinite order.  Then there are constants $c,\varepsilon_0>0$ and an integer $N$ so that  if $0<\varepsilon<\varepsilon_0$ and
\begin{equation}\label{qdata}
\sum_{|\alpha|\le N+1}\|\partial^\alpha_xf\|_{L^2(\Rf\backslash\Kob)}
+\sum_{|\alpha|\le N}\|\partial^\alpha_xg\|_{L^2(\Rf\backslash\Kob)}\le \varepsilon,\end{equation}
then \eqref{q} has a unique solution $u\in C^\infty([0,T_\varepsilon)\times \Rf\backslash \Kob)$ with
\begin{equation}\label{te}
T_\varepsilon=\exp(c/\varepsilon).
\end{equation}
\end{theorem}

For simplicity we have only stated our results for scalar equations.
However, since the arguments only involve the standard vector fields
$\partial_j$ and the generators of spatial rotations $\Omega_{ij}$
described below, they can easily be modified to handle multi-speed
symmetric systems as were treated in earlier related work, such as
\cite{KSS2} and \cite{MS2, MS3}.  

The authors are grateful to Kunio Hidano for helpful comments on an early draft of this paper.

\newsection{Main estimates}

As in \cite{YY} we shall require some estimates which can be thought of as variations of the classical Hardy inequality.  They involve mixed-norms $L^q_rL^q_\omega$ with respect to the standard volume element $r^{n-1}drd\omega$ for ${\mathbb R}^n$, with $d\omega$ denoting the induced Lebesgue measure on $S^{n-1}$.  Thus,
$$\|h\|_{L^p_rL^q_\omega}=\bigl\| \, \|h(r\cdot)\|_{L^q(S^{n-1}, d\omega)}\, \bigr\|_{L^p([0,\infty),r^{n-1}dr)}.$$

\begin{lemma}\label{lemmahar}
If $v\in C^\infty_0({\mathbb R}^n)$, $n\ge3$, and $R>0$, then
\begin{equation}\label{h.1}
R^{1/2}\|v\|_{L^\infty_rL^2_\omega(|x|>R)}\le C\|\,
|x|^{-(n-3)/2}\nabla_xv\|_{L^2(|x|>R)}.
\end{equation}
\end{lemma}

\begin{proof}  The left side of \eqref{h.1} is dominated by
\begin{align*}
R^{1/2}\Bigl(\int_{S^{n-1}}\int_R^\infty &|v(\rho\omega)| \,
|\partial_\rho v(\rho\omega)| \, d\rho d\omega\Bigr)^{1/2}
\\
&\le \Bigl(\int_{S^{n-1}}\int_R^\infty |\partial_\rho v|^2 \rho^2 d\rho
d\omega\Bigr)^{1/4}R^{1/2}\Bigl(\int_R^\infty
\int_{S^{n-1}}|v(\rho\omega)|^2d\omega \tfrac{d\rho}{\rho^2}\Bigr)^{1/4}
\\
&\le \|\, |x|^{-(n-3)/2}\nabla_xv\|_{L^2(|x|>R)}^{1/2}\,
R^{1/4}\|v\|^{1/2}_{L^\infty_rL^2_\omega(|x|>R)},
\end{align*}
which yields \eqref{h.1}.\end{proof}

Using this result we easily get another useful result.

\begin{lemma}\label{lemmahar2}  Let $n\ge3$.  Then if $h\in C^\infty_0({\mathbb R}^n)$
\begin{equation}\label{h.2}
\|h\|_{\dot H^{-1}({\mathbb R}^n)} \le C\|h\|_{L^{2n/(n+2)}(|x|<1)} + C\|\, |x|^{-(n-2)/2} h\|_{L^1_rL^2_\omega(|x|>1)}.
\end{equation}
\end{lemma}

\begin{proof}  Write $h=h_0+h_1$ where $h_0$ equals $h$ when $|x|<1$ and zero otherwise.  By Sobolev's inequality, the $\dot H^{-1}$ norm of $h_0$ is dominated by the first term in the right side of \eqref{h.2}.  As a result, it suffices to show that if $g=h_1$ then
$$\|g\|_{\dot H^{-1}({\mathbb R}^n)}\le C\|\, |x|^{-(n-2)/2} g\|_{L^1_rL^2_\omega({\mathbb R}^n)}.$$
By polarization, this is equivalent to showing that when $v\in \dot H^{1}\cap C^\infty_0$ we have
$$\Bigl|\, \int gv\, dx \, \Bigr|\le C\|\, |x|^{-(n-2)/2}g\|_{L^1_rL^2_\omega}\|\nabla_x v\|_2.$$
By H\"older's inequality, this in turn would follow from
$$\|\, |x|^{(n-2)/2}v\, \|_{L^\infty_rL^2_\omega}\le C\|\nabla_x v\|_2,$$
and since this is a consequence of \eqref{h.1}, the proof is complete.
\end{proof}

We shall also need the ``KSS estimate"\footnotemark\,
from \cite{KSS}.  It concerns solutions of the inhomogeneous wave equation
\begin{equation}\label{inhom}
\begin{cases} (\partial^2_t-\Delta) v(t,x)=G(t,x), \quad (t,x)\in {\mathbb R}_+\times {\mathbb R}^n
\\
v(0,\cd)=\partial_tv(0,\cd)=0,
\end{cases}
\end{equation}
and says that when $n\ge3$ and $S_T=[0,T]\times{\mathbb R}^n$, we have the uniform bounds
\begin{multline}\label{kssest}
\bigl(\log(2+T)\bigr)^{-1/2}\|\langle x
\rangle^{-1/2}v'\|_{L^2(S_T)} + \|\langle
x\rangle^{-1/2-\delta}v'\|_{L^2(S_T)}+\|v'(T,\cd)\|_{L^2(\Rn)}
\\
\le C_\delta\int_0^T\|G(t,\cd)\|_2\, dt,
\end{multline}
if $\delta>0$ is fixed.  Here $\langle x\rangle = (1+|x|^2)^{1/2}$.

\footnotetext{We borrow this name from \cite{A2}, though this is akin
  to earlier estimates from, e.g., \cite{M2}, \cite{Strauss}, and \cite{KPV}. }

Strictly speaking, only the 3-dimensional version was proved in
\cite{KSS}, although a generalization to variable coefficients for all
$n\ge3$ was proved by the second and third authors in \cite{MS}.  We shall
need the latter in \S 4 when we prove almost global results for
quasilinear equations.  To prove the constant coefficient estimate
\eqref{kssest} one first uses the fact that an analog without the
log-factor holds if one only takes the $L^2$ norm over $[0,T]\times
\{x\in {\mathbb R}^n: \, |x|<1\}$.  This can be proven using the
Fourier transform, as was shown in \cite{SS}.  By a scaling argument, for $k=1,2,3,\dots$, one can uniformly control
$$\|\, \langle x\rangle^{-1/2}v'\|_{L^2([0,T]\times \{x: \, |x|\in [2^k,2^{k+1}]\})}$$
by the right side of \eqref{kssest}.  Since, by the standard energy inequality one also has that
$$\|\, \langle x\rangle^{-1/2}v'\|_{L^2([0,T]\times \{x: \, |x|>T\})},$$
is controlled by the right side of \eqref{kssest},
one obtains \eqref{kssest} by summing up the squares of these bounds.

Let us now see how we can use \eqref{h.2} and \eqref{kssest} to prove
a variant of the inequality that 
the first and last authors used in \cite{YY} to obtain 
optimal existence results for quasilinear equations in 3-dimensions:
\begin{theorem}\label{kssyytheorem}
Let $n\ge3$, and suppose that $w$ solves the inhomogeneous wave equation
\begin{equation}\label{2.5}
\begin{cases}
(\partial_t^2-\Delta)w(t,x)=F(t,x), \quad (t,x)\in {\mathbb R}_+\times {\mathbb R}^n
\\ w(0,\cd)=\partial_tw(0,\cd)=0.
\end{cases}
\end{equation}
Then if $\delta>0$ is fixed, we have the uniform bounds
\begin{multline}\label{kssyy}
\bigl(\log(2+T)\bigr)^{-1/2}\|\langle x \rangle^{-1/2}w\|_{L^2(S_T)}
+ \|\langle x\rangle^{-1/2-\delta}w\|_{L^2(S_T)}+\|w(T,\cd)\|_{L^2(\Rn)}
\\
\le C_\delta \int_0^T\|\,
|x|^{-(n-2)/2}F(t,\cd)\|_{L^1_rL^2_\omega(|x|>1)} \:dt+C_\delta \int_0^T
\|F(t,\cd)\|_{L^{2n/(n+2)}(|x|<1)}\:dt.
\end{multline}
\end{theorem}

\begin{proof}  To prove \eqref{kssyy}, we apply \eqref{kssest} to
$$v_j(t,x)=(2\pi)^{-n}\int_{{\mathbb R}^n}e^{ix\cdot\xi}\Hat w(t,\xi) \frac{\xi_j }{|\xi|^2}\, d\xi,\quad j=1,2,\dots,n,$$
where $\Hat w(t,\xi)$ denotes the spatial Fourier transform of $x\to
w(t,x)$.  If we do this, we get \eqref{kssyy} since $\sum_{j=1}^n
\partial_jv_j=i w$ and the $v_j$ solve the inhomogeneous wave equations $(\partial_t^2-\Delta)v_j=G_j$ with zero initial data and forcing terms satisfying 
$$\|G_j(t,\cd)\|_{L^2({\mathbb R}^n)}\le \|F(t,\cd)\|_{\dot
  H^{-1}({\mathbb R}^n)}.$$
\end{proof}

If we use H\"older's inequality, we see that \eqref{kssyy} immediately yields
\begin{multline}\label{kyy}
\bigl(\log(2+T)\bigr)^{-1/2}\|\langle x \rangle^{-1/2}w\|_{L^2(S_T)}
+ \|\langle x\rangle^{-1/2-\delta}w\|_{L^2(S_T)}
\\
\le C_\delta \int_0^T\bigl(\, \|\, 
\langle x\rangle^{-(n-2)/2}(\partial^2_t-\Delta)w(t,\cd)\|_{L^1_rL^2_\omega({\mathbb R}^n)} 
+\|(\partial^2_t-\Delta)w(t,\cd)\|_{L^2({\mathbb R}^n)}\, \bigr)\, dt,
\end{multline}
if, as above, $w$ has vanishing initial data.

Let us now see that this estimate extends to the setting of Dirichlet-wave equations outside of nontrapping obstacles.  To be more specific, we shall assume that $\Kob\subset {\mathbb R}^n$ has smooth boundary and is nontrapping.  There is no loss of generality in also assuming in what follows that $0\in \Kob$ and $\Kob\subset \{x\in {\mathbb R}^n: \, |x|<1\}$.  We shall be concerned with the inhomogeneous wave equations outside $\Kob$:
\begin{equation}\label{2.8}
\begin{cases}
(\partial^2_t-\Delta)w(t,x)=F(t,x), \quad (t,x)\in {\mathbb R}+\times \Kout
\\
w(t,x)=0, \quad x\in \partial\Kob
\\
w(0,\cd)=\partial_tw(0,\cd)=0.
\end{cases}
\end{equation}

The main estimates we require for solutions of this equation are contained in the following:

\begin{theorem}\label{kyob}
Let  $n\ge3$, and let $\Kob\subset \Rn$ and $w$ be as above. For $T>0$, set $S_T^\Kob = [0,T]\times \Kout$.  Then if $\delta>0$
\begin{multline}\label{2.9}
\bigl(\log(2+T)\bigr)^{-1/2}\|\langle x
\rangle^{-1/2}w'\|_{L^2(\Stk)} + \|\langle
x\rangle^{-1/2-\delta}w'\|_{L^2(\Stk)} 
+\|w'(T,\cd)\|_{L^2(\Kout)}
\\
\le C_\delta\int_0^T\|F(t,\cd)\|_{L^2(\Kout)}\, dt,
\end{multline}
and
\begin{multline}\label{2.10}
\bigl(\log(2+T)\bigr)^{-1/2}\|\langle x \rangle^{-1/2}w\|_{L^2(\Stk)}
+ \|\langle x\rangle^{-1/2-\delta}w\|_{L^2(\Stk)}+\|w(T,\cd)\|_{L^2(\Kout)} 
\\
\le C_\delta \int_0^T\bigl(\, \|\, 
\langle x\rangle^{-(n-2)/2}F(t,\cd)\|_{L^1_rL^2_\omega(\Kout)} 
+\|F(t,\cd)\|_{L^2(\Kout)}\, \bigr)\, dt.
\end{multline}
\end{theorem}

The first estimate is due to Keel, Smith and Sogge \cite{KSS}.  See also Metcalfe \cite{Metcalfe}.  Its proof uses the nonobstacle version \eqref{kssest} and a variant of the latter which says that if $v$ is as in \eqref{inhom}, then
\begin{multline*}
(\log(2+T))^{-1/2}\|\, \langle x\rangle^{-1/2}v'\|_{L^2(S_T)}+\|\, \langle x\rangle^{-1/2-\delta}v'\|_{L^2(S_T)}
\\
\le C_\delta \|G\|_{L^2(S_T)}, \quad \text{if } \, \, G(t,x)=(\partial^2_t-\Delta)v(t,x)=0, \, \, \text{for } \, \, |x|>2.
\end{multline*}
See \cite{KSS} and \cite{Metcalfe} for a proof of this estimate,
though it follows more simply from the arguments in \cite{MS, MS3}.

For us, it is convenient to use a slightly stronger version, which says that for $\delta>0$,
\begin{multline}\label{2.11}
(\log(2+T))^{-1/2}\|\, \langle x\rangle^{-1/2}v'\|_{L^2(S_T)}+\|\, \langle x\rangle^{-1/2-\delta}v'\|_{L^2(S_T)}
\\
\le C_\delta \|\langle x\rangle^{1/2+\delta}G\|_{L^2(S_T)}.
\end{multline}
This estimate follows from the nonobstacle version of Proposition 2.2 in \cite{MS3}
(cf. Lemma 4.2 below).  The nonobstacle variant follows from the same multiplier argument that was used in \cite{MS3}, and indeed the arguments are slightly simpler for the Minkowski space version.  When $n\ge4$ we can use this estimate since when $\delta<1/2$ one has that when $j=1,2,\dots,n$,
$$\|\, \langle x\rangle^{1/2+\delta}(\Delta^{-1}\partial_j F)(t,\cd)\|_{L^2(\Rn)}\le C_\delta \|F(t,\cd)\|_{L^2(\Rn)}, \quad \text{if } \, F(t,x)=0, \, |x|\ge 2,$$
by Young's inequality and the fact that the kernel of $\Delta^{-1}\partial_j$ is $O(|x-y|^{-n+1})$.  Therefore, for $n\ge4$, \eqref{2.11}
and
%
the proof of Theorem \ref{kssyytheorem} shows that if $w$ is as in \eqref{2.5}, then
\begin{multline}\label{2.12}
(\log(2+T))^{-1/2}\bigl\|\, \langle x\rangle^{-1/2}w\bigr\|_{L^2(S_T)}+\| \langle x \rangle^{-1/2-\delta}w\|_{L^2(S_T)}
\\
\le C_\delta\|F\|_{L^2(S_T)}, \quad \text{if } \, \, F(t,x)=(\partial^2_t-\Delta)w(t,x)=0, \, \, \text{for } \, |x|>2.
\end{multline}
When $n=3$, in \cite{YY}, Du and Zhou used the sharp Huygens principle and arguments of \cite{KSS} to show that in this case \eqref{2.12} is a consequence of \eqref{kyy}.

The final estimate we require for the proof of \eqref{2.10} is for solutions to Dirichlet-wave equations 
\eqref{2.8}.  Since $\Kob$ is nontrapping one can show that
\begin{equation}\label{2.13}
\sum_{|\alpha|\le 1}\| \partial^\alpha w\|_{L^2([0,T]\times \{x\in \Kout: \, |x|<2\})} \le C\int_0^T\|F(t,\cd)\|_{L^2(\Kout)}\, dt,
\end{equation}
with the constant $C$ being independent of $T$.  This follows from the nontrapping assumption and certain resolvent estimates (see Burq \cite{B}, Theorem 3).  Of course \eqref{2.13} is essentially equivalent to the classical local decay estimates in \cite{M} and \cite{MRS}.

Using \eqref{2.13} we see that the analog of \eqref{2.10} holds if the norms in the left are taken over $[0,T]\times \{x\in \Kout: \, |x|<2\}$.  Let us now present the standard argument from \cite{SS} and \cite{KSS} that shows that \eqref{kssest}, \eqref{2.12} and \eqref{2.13} control the remaining piece where $|x|>2$.

To do this, we choose $\rho\in C^\infty(\Rn)$ satisfying $\rho(x)=1$, $|x|\ge 2$, and $\rho(x)=0$, $|x|\le 1$.  If we set $\tilde w(t,x)=\rho(x)w(t,x)$ then of course $w(t,x)=\tilde w(t,x)$ when $|x|>2$.  Also, $\tilde w$ solves the Minkowski space equation $(\partial_t^2-\Delta)\tilde w= \rho F -2\nabla_x\rho \cdot \nabla_x w-(\Delta \rho)w$.  Note that the last two terms vanish when $|x|>2$.  Therefore, if we apply \eqref{kssyy} and \eqref{2.12} to $\tilde w$ we deduce that
\begin{multline*}
(\log(2+T))^{-1/2}\|\, \langle x\rangle^{-1/2}w\|_{L^2([0,T]\times \{x\in \Kout: \, |x|>2\})}
\\+\|\, \langle x\rangle^{-1/2-\delta}w\|_{L^2([0,T]\times \{x\in \Kout: \, |x|>2\})}
\le C_\delta \int_0^T\|\, \langle x\rangle^{-(n-2)/2}F(t,\cd)\|_{L^1_rL^2_\omega(\Kout)}\, dt
\\
+C_\delta \sum_{|\alpha|\le 1}\|\partial_x^\alpha w\|_{L^2([0,T]\times \{x\in \Kout: \, |x|<2\})}.
\end{multline*}
Since the last term in the right can be controlled using \eqref{2.13}, we conclude that we have the desired bounds in the region where $|x|>2$, which completes the proof of \eqref{2.10}.

\medskip

We also require a variant of \eqref{2.9}-\eqref{2.10} which involves the following ``admissible" vector fields
$$\{Z\}=\{\, \partial_0,\dots,\partial_n,\Omega_{ij}: \, 1\le i<j\le n\, \},$$
where,  
$\Omega_{ij}=x_i\partial_j-x_j\partial_i, \quad 1\le i<j\le n$,
are the generators of spatial rotations.  If one uses standard elliptic regularity arguments, as in \cite{KSS}, one finds that Theorem \ref{kyob} yields the following useful result.

\begin{corr}\label{cork}  Let  $n\ge 3$, and let $\Kob\subset \Rn$ and $w$ be as above.  Then if $\delta>0$ and $N=1,2,\dots$ are fixed
\begin{align}\label{2.14}
&\sum_{|\alpha|\le N}\Bigl[
\bigl(\log(2+T)\bigr)^{-1/2}\|\langle x
\rangle^{-1/2}Z^\alpha w'\|_{L^2(\Stk)} + \|\langle
x\rangle^{-1/2-\delta}Z^\alpha w'\|_{L^2(\Stk)} 
\\
&\qquad \qquad \qquad+\|Z^\alpha w'(T,\cd)\|_{L^2(\Kout)}\Bigr]
\notag
\\
&\le C_\delta\int_0^T\sum_{|\alpha|\le N}\| Z^\alpha F(t,\cd)\|_{L^2(\Kout)}\, dt
\notag
\\
&\qquad \qquad \qquad+C\sum_{|\alpha|\le N-1}\Bigl[\, \sup_{0<s<T}\|Z^\alpha F(s,\cd)\|_{L^2(\Kout)}+\|Z^\alpha F\|_{L^2(S^\Kob_T)}\Bigr]
\notag
\end{align}
and
\begin{align}\label{2.15}
&\sum_{|\alpha|\le N}\Bigl[
\bigl(\log(2+T)\bigr)^{-1/2}\|\langle x \rangle^{-1/2} Z^\alpha w\|_{L^2(\Stk)}
+ \|\langle x\rangle^{-1/2-\delta}Z^\alpha w\|_{L^2(\Stk)}
\\
&\qquad \qquad \qquad+\|Z^\alpha w(T,\cd)\|_{L^2(\Kout)} \Bigr],
\notag
\\
&\le C_\delta \int_0^T\sum_{|\alpha|\le N}\bigl(\, \|\, 
\langle x\rangle^{-(n-2)/2}Z^\alpha  F(t,\cd)\|_{L^1_rL^2_\omega(\Kout)} 
+\|Z^\alpha F(t,\cd)\|_{L^2(\Kout)}\, \bigr)\, dt\notag
\\
&\qquad \qquad \qquad+C\sum_{|\alpha|\le N-1}\Bigl[\, \sup_{0<s<T}\|Z^\alpha F(s,\cd)\|_{L^2(\Kout)}+\|Z^\alpha F\|_{L^2(S^\Kob_T)}\Bigr]
\notag.
\end{align}
\end{corr}

\newsection{The Strauss conjecture in 4-dimensions}

To prove Theorem \ref{theorem1} we shall require a couple of Sobolev
estimates involving the vector fields
$\{Z\}=\{\partial_0,\dots,\partial_4,\Omega_{ij}: \, 1\le i<j\le 4\}$ introduced before.

\begin{lemma}\label{lemma3.1}  If 
$h\in C^\infty(\Rf)$ and $R>1$ then
\begin{equation}\label{31}
\|h\|_{L^\infty(\{|x|\in [R,R+1]\})}\le CR^{-3/2}\sum_{|\alpha|\le 2,
j\le 1}\|\Omega^\alpha \nabla^j_x h\|_{L^2(\{|x|\in [R-1,R+2]\})},
\end{equation}
and
\begin{equation}\label{32}\|h\|_{L^4(\{|x|\in [R,R+1]\})}\le CR^{-3/4}\sum_{|\alpha|\le
1}\|Z^\alpha h\|_{L^2(\{|x|\in [R-1,R+2]\})}.\end{equation}
\end{lemma}

\begin{proof}  To prove \eqref{31}, we first use polar coordinates and Sobolev's lemma for $S^3$ to obtain
$$\sup_{\omega\in S^3}|h(r\omega)|\le C\sum_{|\alpha|\le 2}\Bigl(\, \int_{S^3}|\Omega^\alpha h(r\omega)|^2 d\omega\, \Bigr)^{1/2},$$
where, as before, $d\omega$ is the standard volume element on $S^3$.  From this and an application of the 1-dimensional Sobolev lemma, we conclude that the left side of \eqref{31} is dominated by 
$$\sum_{|\alpha|\le 2, j\le 1}\Bigl(\, \int_{R-1}^{R+2}\int_{S^3}|\Omega^\alpha \partial^j_r h(r\omega)|^2 \, dr d\omega\, \Bigr)^{1/2},
$$
which leads to \eqref{31} due to the fact that that the volume element for $\Rf$ is a constant times $r^3dr d\omega$.  Estimate \eqref{32} similarly follows from the $L^4$-Sobolev estimates for $\R\times S^3$. \end{proof}

We now have all the ingredients we require for the proof of Theorem \ref{theorem1}.  As in \cite{KSS}, it is convenient to show that one can solve an equivalent nonlinear equation with zero data to avoid having to deal with issues regarding compatibility conditions for the data.  

To make the reduction, we first note that if the data satisfies \eqref{Stdata} with $\varepsilon$ small, we can find a local solution $u$ to $\square u=F_\sigma(u)$ in $0<t<1$.  Our assumptions on the data and the finite propagation speed for $\square$ furthermore ensure that if $\varepsilon>0$ is small enough this local solution will satisfy
\begin{equation}\label{33}
\sup_{0\le t\le1}\sum_{|\alpha|\le 2, j\le 1}\|Z^\alpha \nabla^j u(t,\cd)\|_{L^2(\Kfout)}\le C\varepsilon,\end{equation}
for some constant $C$.  Since we are assuming that the data are
compactly supported, we also have that $u(t,x)$ vanishes for large $x$
when $0\le t\le 1$.  

Using this local solution we can set up the iteration argument that will be used to prove Theorem \ref{theorem1}.  We first fix a bump function $\eta\in C^\infty(\R)$ satisfying $\eta(t)=1$ if $t\le 1/2$ and $\eta(t)=0$ if $t>1$.  Set
$$u_0=\eta u;$$
then
$$\square u_0=\eta F_\sigma(u)+[\square,\eta]u.$$
So $u$ solves $\square u= F_\sigma(u)$ if and only if $w=u-u_0$ solves
\begin{equation}\label{34}
\begin{cases}
\square w=(1-\eta)F_\sigma(u_0+w)-[\square,\eta]u
\\
w|_{\partial\Kob}=0
\\
w(0,x)=\partial_tw(0,x)=0.
\end{cases}
\end{equation}

We shall solve this equation by iteration.  Set $w_0\equiv 0$ and define $w_k$, $k=1,2,3,\dots$  inductively by requiring that
\begin{equation}\label{35}
\begin{cases}
\square w_k=(1-\eta)F_\sigma(u_0+w_{k-1})-[\square,\eta]u
\\
w_k|_{\partial\Kob}=0
\\
w_k(0,x)=\partial_tw_k(0,x)=0.
\end{cases}
\end{equation}

Our aim then is to show that if the constant $\varepsilon$ in \eqref{Stdata} is small then so is
\begin{multline*}M_k(T)=\sum_{|\alpha|\le 2}\Bigl(\|\langle x\rangle^{-1/2-\delta}Z^\alpha
w_k\|_{L^2(S_T^\Kob)}+\sup_{0<t<T}\|Z^\alpha w_k(t,\cd)\|_2\Bigr)
\\
+\sum_{|\alpha|\le 2}\Bigl(\|\langle x\rangle^{-1/2-\delta}Z^\alpha
w_k'\|_{L^2(S_T^\Kob)}+\sup_{0<t<T}\|Z^\alpha w_k'(t,\cd)\|_2\Bigr),
\end{multline*}
for every $k=1,2,3,\dots$ and $T>0$,
with $\delta=3(\sigma-2)/4$.  Recall that we are assuming that $\sigma>2$.

To estimate $M_k(T)$, we use \eqref{2.15} to estimate its first part and \eqref{2.14} to estimate its second part.  We conclude that there is a constant $C_0$, depending on the bounds in \eqref{33}, and a constant $C$, depending on the constant in \eqref{2.14} so that for $k=1,2,3,\dots$ we have
\begin{align}\label{36}
M_k&(T)
\\
&\le C_0\varepsilon \notag
\\
& +C\int_0^T\sum_{|\alpha|\le 2}
\bigl(\, \| \langle x\rangle^{-1}Z^\alpha F_\sigma(u_0+w_{k-1})(t,\cd)\|_{L^1_rL^2_\omega}+
\|Z^\alpha F_\sigma(u_0+w_{k-1})(t,\cd)\|_{L^2}\, \bigr)\, dt
\notag
\\
&+C\sum_{|\alpha|\le 1}\Bigl[ \, \sup_{0<s<T}\|Z^\alpha F_\sigma(u_0+w_{k-1})(s,\cd)\|_{L^2(\Kfout)}+\|Z^\alpha F_\sigma(u_0+w_{k-1})\|_{L^2(S^\Kob_T)}\Bigr]\notag
\\
&=C_0\varepsilon + I + II + III + IV.\notag
\end{align}

We claim that  if $\varepsilon>0$ is small enough then we have
\begin{equation}\label{37}M_k(T)\le 2C_0\varepsilon, \quad T>0,\end{equation}
for every $k$.  Since $w_{0}\equiv 0$ and \eqref{33} is valid, the first term, $M_1(T)$, in the induction satisfies these bounds if $\varepsilon$ is small.  Therefore, our task will be to show that if \eqref{37} is valid with $k$ replaced by $k-1$, then \eqref{37} must hold.

Let us start by handling the main term in the right side of \eqref{36}, which is $I$.  We note that our assumptions regarding $F_\sigma$ imply that
\begin{align}\label{38}
&\sum_{|\alpha|\le 2}|Z^\alpha F_\sigma(v(t,x))|
\\
&\quad\le C|v|^{\sigma-2}\Bigl(\, |v|\sum_{|\alpha|\le 2}|Z^\alpha v|+\sum_{|\alpha|\le 1}|Z^\alpha v|^2\Bigr)\notag
\\
&\quad\le C\Bigl(\sum_{|\alpha|\le 2, j\le 1} \|Z^\alpha \nabla^j v(t,\cd)\|_{L^2(\Kfout)}\Bigr)^{\sigma-2} \, \langle x\rangle^{-\frac{3(\sigma-2)}2}\Bigl(\, |v|\sum_{|\alpha|\le 2}|Z^\alpha v|+\sum_{|\alpha|\le 1}|Z^\alpha v|^2\Bigr).\notag
\end{align}
To handle the contribution of the last factor to $I$ when $v=w_{k-1}+u_0$, we note that Sobolev estimates for $S^3$ yield
\begin{align*}
\bigl\|\, v(t,r\cd)\sum_{|\alpha|\le 2}|Z^\alpha &v(t,r\cd)|\,
\bigr\|_{L^2(S^3)}^2+
\bigl\|\sum_{|\alpha|\le 1}|Z^\alpha v(t,r\cd)|^2\bigr\|_{L^2(S^3)}
\\
&\le \|v(t,r\cd)\|_{L^\infty(S^3)}\sum_{|\alpha|\le 2}\|Z^\alpha v(t,r\cd)\|_{L^2(S^3)}+
\sum_{|\alpha|\le 1}\|Z^\alpha v(t,r\cd)\|^2_{L^4(S^3)}
\\
&\le C\sum_{|\alpha|\le 2}\|Z^\alpha v(t,r,\cd)\|_{L^2(S^3)}^2.
\end{align*}
By combining the last two inequalities, one sees that
\begin{align*}
I
&\le C(M_{k-1}(T)+\varepsilon)^{\sigma-2}\int_0^T\sum_{|\alpha|\le
  2}\Bigl(\|\langle x\rangle^{-1/2-\delta}Z^\alpha
w_{k-1}\|^2_{L^2(\Kfout)}\\
&\qquad\qquad\qquad\qquad\qquad\qquad\qquad\qquad\qquad\qquad
+\|\langle x\rangle^{-1/2-\delta}Z^\alpha u_0\|^2_{L^2(\Kfout)}\Bigr)\, dt
\\
&\le C(M_{k-1}(T)+\varepsilon)^\sigma.
\end{align*}

To handle $II$, we note that \eqref{31} and \eqref{32} imply that
\begin{multline*}
\bigl\|\, |v(t,\cd)|^{\sigma-1}\sum_{|\alpha|\le 2}|Z^\alpha v(t,\cd)|\, \bigr\|_{L^2}
+\bigl\|\, |v(t,\cd)|^{\sigma-2}\sum_{|\alpha|\le 1}|Z^\alpha v(t,\cd)|^2\bigr\|_{L^2}
\\
\le C\sum_{|\alpha|\le 2, j\le 1}\|\, \langle
x\rangle^{-3/4-\delta}Z^\alpha \nabla^j v(t,\cd)\|_2^2
\sum_{|\alpha|\le 2, j\le 1}\|Z^\alpha \nabla^j_xv(t,\cd)\|^{\sigma-2}_2.
\end{multline*}
Consequently, the arguments used for $I$, imply that
\begin{align*}
II&\le C(M_{k-1}(T)+\varepsilon)^{\sigma-2}\sum_{|\alpha|\le 2, j\le
  1}\Bigl(\|\langle x\rangle^{-3/4-\delta}Z^\alpha \nabla^j w_{k-1}\|^2_{L^2(S^\Kob_T)}
\\&\qquad\qquad\qquad\qquad\qquad\qquad\qquad\qquad\qquad\qquad
+\|\langle x\rangle^{-3/4-\delta} Z^\alpha \nabla^j u_0\|^2_{L^2(S^\Kob_T)}\Bigr)
\\
&\le C(M_{k-1}(T)+\varepsilon)^\sigma.
\end{align*}

Since similar arguments show that $III$ and $IV$ also enjoy these bounds we conclude that
$$M_k(T)\le C_0\varepsilon + (M_{k-1}(T)+\varepsilon)^\sigma,$$
which in turn implies \eqref{37} by induction.

To finish and show that $u_k=w_k+u_0$, $k=1,2,3,\dots$ converges to a solution $u$ of \eqref{Str} it suffices to show that
$$
A_k(T)=\sum_{j\le 1}\Bigl(\|\, \langle x\rangle^{-1/2-\delta}\nabla^j (u_k-u_{k-1})\, \|_{L^2(S^\Kob_T)}+\sup_{0<t<T}\| \nabla^j(u_k-u_{k-1})(t,\cd)\|_2\Bigr)$$
tends to zero as $k\to \infty$.  Since $|F_\sigma(v)-F_\sigma(w)|\le C|v-w|(\, |v|^{\sigma-1}+|w|^{\sigma-1})$, the proof of \eqref{36} can be adapted to show that
$$A_k(T)\le CA_{k-1}(T)\, \bigl(M_{k-1}(T)+M_{k-2}(T)\bigr)^{\sigma-1},$$
which, by \eqref{37}, implies that $A_k(T)\le \tfrac12 A_{k-1}(T)$ if $\varepsilon>0$ is small.  Since $A_1$ is finite, the claim follows.

This completes the proof of Theorem \ref{theorem1}.

\newsection{Almost global existence in 4-dimensions}

In this section we shall prove Theorem \ref{qthm}.  In addition to the
estimates collected in \S 2, we shall need a slight variant of Lemma
\ref{lemmahar} as well as variable coefficient versions of the ``KSS
estimates" \eqref{kssest} and \eqref{2.14} that are due to the second
and third authors \cite{MS}.

The first result mentioned is the following:

\begin{lemma}\label{lemma40}  Suppose that $g,h\in C^\infty_0(\Rn)$,
  $n\ge 4$.\footnotemark\, Then if $R>0$
\begin{multline}\label{4.0}
\|gh\|_{L^2(\{|x|\in [R,2R]\})}
\\
\le C\sum_{|\alpha|\le \frac{n+1}2}\|\, |x|^{-(n-2)/4}\Omega^\alpha
g\|_{L^2(\{|x|\in [R,2R]\})}
\|\, |x|^{-(n-2)/4}h'\|_{L^2(\{|x|>R\})}.
\end{multline}
\end{lemma}

\footnotetext{This lemma can also be recovered in $n=3$, but it
  requires a straightforward modification of the proof of Lemma \ref{lemmahar}.}

\begin{proof}  By H\"older's inequality and Sobolev's lemma for $S^{n-1}$ the left side of \eqref{4.0} is dominated by
\begin{multline*}
\|g\|_{L^2_rL^\infty_\omega(\{|x|\in [R,2R]\})}\|h\|_{L^\infty_rL^2_\omega(\{|x|\in [R,2R]\})}
\\
\le C\sum_{|\alpha|\le \frac{n+1}2}\|\Omega^\alpha g\|_{L^2(\{|x|\in [R,2R]\})}\|h\|_{L^\infty_rL^2_\omega(\{|x|\in [R,2R]\})}.
\end{multline*}
The result now follows from an application of \eqref{h.1}.
\end{proof}

The variable coefficient variants of the ``KSS estimate" 
involve solutions $\phi \in C^\infty(\R_+\times {\mathbb R}^n\backslash \Kob)$ of the Dirichlet-wave equation
\begin{equation}\label{4.1}
\begin{cases}
\square_h \phi = F
\\
\phi|_{\partial \Kob}=0
\\
\phi|_{t=0}=\partial_t\phi|_{t=0}=0,
\end{cases}
\end{equation}
where $\Kob\subset \Rn$ is now assumed to be a star-shaped obstacle containing the origin and
\begin{equation}\label{4.2}
\square_h\phi = (\partial_t^2-\Delta)\phi + \sum_{\alpha, \beta=0}^n h^{\alpha\beta}(t,x)\partial_\alpha\partial_\beta\phi.
\end{equation}
We shall assume that the $h^{\alpha\beta}$ satisfy the symmetry  conditions
\begin{equation}\label{4.3}
h^{\alpha\beta}=h^{\beta\alpha},
\end{equation}
as well as the size conditions
\begin{equation}\label{4.4}
|h|=\sum_{\alpha,\beta=0}^n |h^{\alpha\beta}(t,x)|\le \delta,
\end{equation}
with $\delta>0$ being a small number.

The estimates we require then are Lemma 5.2 and Lemma 5.3 from \cite{MS}.
The first is the following:

\begin{lemma}
\label{lemma4.1}
Suppose that $\mathcal{K}$ is as above and $n\ge 3$. Let $\phi\in C^\infty(\R_+\times\extn)$ be a
solution of \eqref{4.1}. Suppose that $h^{\alpha\beta}$ satisfies \eqref{4.3} and \eqref{4.4}
for a small choice of $\delta$. Then, if $\varepsilon>0$ and $N=0,1,2,\dots$ are fixed there is a constant $C$ so that for any $T>0$
\begin{gather}
(\log(2+T))^{-1/2} \sum_{|\mu|\le N} \|\langle x\rangle^{-1/2}\partial^\mu\phi'\|_{L^2(\Stk)} \label{4.5} \\
+\sum_{|\mu|\le N} \Bigl(\|\langle x\rangle^{-1/2-\varepsilon} \partial^\mu \phi'\|_{L^2(\Stk)} +\|\partial^\mu\phi'(T,\cd)\|_{L^2(\Kout)}\Bigr)\nonumber \\
\le  C\sum_{j,k\le N}\left(
\int_0^T \int_\extn \left(|\nabla \partial_t^j \phi|+\frac{|\partial_t^j \phi|}{r}\right)
|\Box_h \partial_t^k \phi|\:dx\:dt\right)^{1/2} \nonumber \\
+C\sum_{j,k\le N} \left[\int_0^T \int_{\extn}
\left(|\partial h|+\frac{|h|}{r}\right) |\nabla \partial_t^j
\phi| \left(|\nabla \partial_t^k \phi|+\frac{|\partial_t^k \phi|}{r}\right)
\:dx\:dt\right]^{1/2} \nonumber \\
+ C\sum_{|\mu|\le N-1} \|\Box \partial^\mu \phi\|_{L^2(\Stk)}
+ C\sum_{|\mu|\le N-1} \|\Box \partial^\mu \phi(T,\cd)\|_{L^2(\Kout)}
. \nonumber
\end{gather}
\end{lemma}

Here, we are using the notation $|\partial
h|=\sum_{\alpha,\beta,\gamma=0}^n |\partial_\gamma h^{\alpha\beta}(t,x)|$.
Note that in \eqref{4.5} $r$ is bounded from below as $0\in \Kob$.

The other estimate we require is the following

\begin{lemma}
\label{lemma4.2}
Suppose that $\mathcal{K}$ is as above and $n\ge 3$. Let $\phi\in C^\infty(\R_+\times\extn)$ be a
solution of \eqref{4.1}. Suppose that $h^{\alpha\beta}$ satisfies \eqref{4.3} and \eqref{4.4}
for a small choice of $\delta$. Then, if $\varepsilon>0$ and $N=0,1,2,\dots$ are fixed there is a constant $C$ so that for any $T>0$
\begin{gather}
(\log(2+T))^{-1/2} \sum_{|\mu|\le N} \|\langle x\rangle^{-1/2}Z^\mu \phi'\|_{L^2(\Stk)}
\label{4.6} \\
+\sum_{|\mu|\le N}\Bigl( \|\langle x\rangle^{-1/2-\varepsilon} Z^\mu \phi'\|_{L^2(\Stk)}+\|Z^\mu\phi'(T,\cd)\|_{L^2(\Kout)}\Bigr) \nonumber \\
\le
C\sum_{|\mu|,|\nu|\le N}\left(
\int_0^T \int_\extn \left(|\nabla Z^\mu \phi|+\frac{|Z^\mu \phi|}{r}\right)
|\Box_h Z^\nu \phi|\:dx\:dt\right)^{1/2} \nonumber \\
+C\sum_{|\mu|,|\nu|\le N}
\left[\int_0^T \int_{\extn} \left(|\partial h|+\frac{|h|}{r}\right)
|\nabla Z^\mu \phi| \left(|\nabla Z^\nu \phi|+\frac{|Z^\nu \phi|}{r}\right)
\:dx\:dt\right]^{1/2} \nonumber \\
+ C\sum_{|\mu|\le N+1} \|\partial_{x}^\mu \phi'\|_{L^2([0,T]\times\{|x|<1\})}.\nonumber
\end{gather}
\end{lemma}

Let us now present the existence argument.  As before, we first note that we can find a local solution $u$ to our equation $\square u=Q(u,u',u'')$ in $0<t<1$ if the data satisfies \eqref{qdata} with $\varepsilon>0$ small and $N=21$.\footnotemark\,  Our assumptions on the data ensure that
\begin{equation}\label{4.7}\sup_{0\le t\le 1}\sum_{|\alpha|\le 22}\|Z^\alpha u(t,\cd)\|_{L^2(\Kfout)}\le C\varepsilon,
\end{equation}
for some constant $C$.  Since we are assuming that the data are compactly supported, we also have that $u(t,x)$ vanishes for large $x$ when $0\le t\le 1$.

\footnotetext{This choice of $N$ is certainly not sharp.}

As before, we shall use this local solution to set up the iteration argument that will be used to prove Theorem \ref{qthm}.  We fix a bump function $\eta\in C^\infty(\R)$ satisfying $\eta(t)=1$ if $t\le 1/2$ and $\eta(t)=0$ if $t>1$, and set $u_0=\eta u$.  Then $\square u_0 = \eta \square u + [\square,\eta]u$, and so $u$ will solve $\square u = Q(u,u',u'')$ for $0<t<T_\varepsilon$, where $T_\varepsilon$ is as in \eqref{te}, if and only if $w=u-u_0$ solves
\begin{equation}\label{4.8}
\begin{cases}
\square w = (1-\eta)Q(w+u_0,(w+u_0)', (w+u_0)'')-[\square,\eta]u
\\
w|_{\partial\Kob}=0
\\
w(0,x)=\partial_tw(0,x)=0.
\end{cases}
\end{equation}

We shall solve this equation by iteration.  As before, set $w_0\equiv 0$ and define $w_k$, $k=1,2,3,\dots$ inductively by requiring that
\begin{equation}\label{4.9}
\begin{cases}
\square w_k=(1-\eta)
Q_k
-[\square,\eta]u
\\
w_k|_{\partial\Kob}=0
\\
w_k(0,x)=\partial_tw_k(0,x)=0,
\end{cases}
\end{equation}
where, for shorthand, we are taking
\begin{equation}\label{4.10}
Q_k=Q(u_{k-1}, u_{k-1}', u_k''),
\end{equation}
with, as before, $u_j=w_j+u_0$.

Our aim then is to show that if the constant $\varepsilon$ in \eqref{qdata} is small then so is
\begin{multline}
M_k(T)\label{4.11}
=\sum_{|\alpha|\le 20}\Bigl( \sup_{0\le t\le T}\|\partial^\alpha
w'_k(t,\cd)\|_2
+(\log(2+T))^{-\tfrac12}\|\langle x\rangle^{-\tfrac12}\partial^\alpha w_k'\|_{L^2(\Stk)} \\+ \|\langle x\rangle^{-\tfrac32}\partial^\alpha w_k'\|_{L^2(\Stk)}\Bigr) 
\\
+\sum_{|\alpha|\le 19}\Bigl(\sup_{0\le t\le T} \|Z^\alpha w'_k(t,\cd)\|_2+(\log(2+T))^{-\tfrac12}\|\langle x\rangle^{-\tfrac12}Z^\alpha w_k'\|_{L^2(\Stk)} \\+ \|\langle x\rangle^{-\tfrac32}Z^\alpha w_k'\|_{L^2(\Stk)}\Bigr)
\end{multline}
and
\begin{equation}\label{4.12}
N_k(T)=\sum_{|\alpha|\le 13}(\log(2+T))^{-\tfrac12}\|\langle x\rangle^{-\tfrac12}Z^\alpha w_k\|_{L^2(\Stk)},
\end{equation}
provided that $T\le T_\varepsilon$.

If we use \eqref{4.5}-\eqref{4.6} we conclude that there must be an absolute constant $C_0$ so that
\begin{multline}
\label{4.13}
M_k(T)\le 
C\sum_{|\mu|,|\nu|\le 20} \Bigl(\int_0^T \int_{\Kfout}
\Bigl(|\nabla \partial^\mu w_k|+\frac{|\partial^\mu w_k|}{r}\Bigr) |\partial^\nu \Box_h w_k|\:dx\:dt\Bigr)^{1/2}
\\
\noalign{\vskip8pt}
+ C\sum_{|\mu|,|\nu|\le 20}
\Bigl(\int_0^T \int_{\Kfout} \Bigl(|\nabla \partial^\mu w_k| +\frac{|\partial^\mu w_k|}{r}\Bigr)
|[\partial^\nu,\Box_h]w_k|\:dx\:dt\Bigr)^{1/2}
\\
\noalign{\vskip8pt}
+C\sum_{|\mu|,|\nu|\le 20} \Bigl[\int_0^T \int_{\Kfout} \Bigl(|\partial h|+\frac{|h|}{r}\Bigr) |\nabla
\partial^\mu w_k| \Bigl(|\nabla \partial^\nu w_k|+\frac{|\partial^\nu w_k|}{r}\Bigr)\:dx\:dt\Bigr]^{1/2}
\\
\noalign{\vskip8pt}
+C \sum_{|\mu|,|\nu|\le 19} \Bigl(\int_0^T \int_{\Kfout} \Bigl(|\nabla Z^\mu w_k| +\frac{|Z^\mu w_k|}{r}\Bigr)
|Z^\nu \Box_h w_k|\:dx\:dt
\Bigr)^{1/2}
\\
\noalign{\vskip8pt}
+C\sum_{|\mu|,|\nu|\le 19} \Bigl(\int_0^T \int_{\Kfout} \Bigl(|\nabla Z^\mu w_k|+\frac{|Z^\mu w_k|}{r}\Bigr)
|[Z^\nu,\Box_h]w_k|\:dx\:dt\Bigr)
^{1/2}
\\
\noalign{\vskip8pt}
+C\sum_{|\mu|,|\nu|\le 19} \Bigl[\int_0^T \int_{\Kfout}
\Bigl(|\partial h|+\frac{|h|}{r}\Bigr) |\nabla Z^\mu
w_k|\Bigl(|\nabla Z^\nu w_k|+\frac{|Z^\nu w_k|}{r}\Bigr)\:dx\:dt\Bigr]^{1/2}
\\
\noalign{\vskip8pt}
+C\sup_{0\le t\le T}\Bigl[\sum_{|\mu|\le 19} \|\partial_{t,x}^\mu \Box w_k(t,\cd)\|_2 \Bigr]
+C\sum_{|\mu|\le 19} \|\partial_{t,x}^\mu \Box w_k\|_{L^2(\Stk)}
\\
\noalign{\vskip8pt}
=
I + II + III +\cdots + VIII.
\end{multline}
Here, we set $h^{\gamma\delta} = -(1-\eta) B^{\gamma\delta}(u_{k-1},u'_{k-1})$. 

If we use \eqref{2.15} and \eqref{4.7}, we find that there is a uniform constant $B_0$ so that
\begin{multline}\label{4.14}
N_k(T)\le B_0\varepsilon+C\sum_{|\mu|\le 13}\int_0^T\bigl(\|\langle x\rangle^{-1}Z^\mu Q_k\|_{L^1_rL^2_\omega}+\|Z^\mu Q_k\|_{L^2_x}\bigr)\, dt
\\
+C\sum_{|\mu|\le 12}\Bigl(\sup_{0<s<T}\|Z^\mu Q_k(s,\cd)\|_{L^2_x}
+\|Z^\mu Q_k\|_{L^2(\Stk)}\Bigr).
\end{multline}

Taking $A_0$ to be a constant chosen sufficiently large, depending on
the constants in \eqref{4.7} and \eqref{4.13} as well as $B_0$ from \eqref{4.14},
we claim that if $\varepsilon$ is small enough and $T_\varepsilon = \exp(c/\varepsilon)$, with $c$ small and fixed, then
\begin{equation}\label{4.15}
M_k(T)+N_k(T)\le 4A_0\varepsilon, \quad 0\le T<T_\varepsilon, \quad k=1,2,3\dots .
\end{equation}
Since $w_0\equiv 0$ and \eqref{4.7} holds, the first term in the induction satisfies these bounds for all $T$ if $\varepsilon$ is small.  Therefore, our task will be to show that if \eqref{4.15} is valid with $k$ replaced by $k-1$, then \eqref{4.15} must hold.

Let us first bound $M_k(T)$.  To do this, let us start out by considering the terms in \eqref{4.13} involving the vector fields $\{Z\}$, since they are the most delicate to handle.  In order to estimate the first two of these, $IV$ and $V$, we shall use the fact
\begin{align}\label{4.17}
\sum_{|\mu|\le 19}\Bigl(\, |Z^\mu \square_h w_k|&+|\, [Z^\mu,\square_h]w_k|\, \Bigr)
\\
&\le C\sum_{|\mu|\le 9, a\le 1}|Z^\mu \partial^a u_{k-1}|\sum_{|\nu|\le 19,b\le 1}|Z^\nu \partial^b u_k| \notag
\\
&+C\sum_{|\mu|\le 10, a\le1}|Z^\mu \partial^a u_k|\sum_{|\nu|\le 19, b\le 1}|Z^\nu\partial^b u_{k-1}|
\notag
\\
&+C\sum_{|\mu|\le 9, a\le1}|Z^\mu \partial^a u_{k-1}|\sum_{|\nu|\le 19, b\le 1}|Z^\nu \partial^b u_{k-1}| \notag
\\
&+C |u_{k-1}| \sum_{|\mu|\le 19} |Z^\mu u_0''|
+ C\mathbf{1}_{[0,1]}(t)\sum_{|\mu|\le 20} |Z^\mu u|.
\notag
\end{align}
Note also that, by Hardy's inequality, we can control the other factor in $IV$ and $V$ by noticing that for $0<t<T$ 
\begin{equation*}
\sum_{|\mu|\le 19}\Bigl(\|\nabla_{t,x}Z^\mu w_k(t,\cd)\|_2+\|r^{-1}Z^\mu w_k(t,\cd)\|_2\Bigr)
\le C\sum_{|\mu|\le 19}\|Z^\mu w'_k(t,\cd)\|_2 \le C M_k(T).
\end{equation*}
Therefore, by the Schwarz inequality, we are led to consider the $L^1_tL^2_x$ norm of the terms in the right side of \eqref{4.17}.  If we fix $t>1$ and take the $L^2_x$ norm over a region where $|x|\approx R$, we can use \eqref{31} to conclude that all the resulting terms except for the ones where $b=0$ are controlled by
\begin{align*}
&\sum_{|\alpha|\le 13}\|\langle x\rangle^{-1/2}Z^\alpha u_{k-1}\|_{L^2(|x|\approx R)}
\sum_{|\alpha|\le 19}\|\langle x\rangle^{-1/2}Z^\alpha u'_k\|_{L^2(|x|\approx R)}
\\
&+\sum_{|\alpha|\le 13}\bigl(\|\langle x\rangle^{-1/2}Z^\alpha u_k\|_{L^2(|x|\approx R)}
+\|\langle x\rangle^{-1/2}Z^\alpha u'_k\|_{L^2(|x|\approx R)}\bigr)
\\
&\qquad\qquad\qquad\qquad\times
\sum_{|\alpha|\le 19}\|\langle x \rangle^{-1/2}Z^\alpha u'_{k-1}\|_{L^2(|x|\approx R)}
\\
&+\sum_{|\alpha|\le 13}\|\langle x\rangle^{-1/2}Z^\alpha u_{k-1}\|_{L^2(|x|\approx R)} \sum_{|\alpha|\le 19}\|\langle x\rangle^{-1/2}Z^\alpha u'_{k-1}\|_{L^2(|x|\approx R)}.
\end{align*}
Recall that $u_j=w_j$ if $t\ge 1$.  
If we sum over dyadic $R=2^j$ and apply the Schwarz inequality in $t$, 
we conclude that, when $T\le T_\varepsilon$, the contribution to $IV$ and $V$ of the terms in \eqref{4.17} with $b=1$ is controlled by
\begin{multline}\label{4.18}
C(M_k(T))^{1/2}(\log(2+T))^{1/2}
\\ \times
\bigl[ (N_{k-1}(T)M_k(T))^{1/2}+([N_k(T)+M_k(T)]M_{k-1}(T))^{1/2} +(N_{k-1}(T)M_{k-1}(T))^{1/2}\bigr]
\\
+C(M_k(T))^{1/2} \varepsilon^{1/2}(\varepsilon+M_k(T)+M_{k-1}(T)+N_{k-1}(T)+N_k(T))^{1/2}
\\
\le C\bigl[(\varepsilon \log(2+T_\varepsilon))^{1/2}+\varepsilon^{1/2}\bigr]\bigl(M_k(T)+N_k(T\bigr))+C\varepsilon^{3/2}, \end{multline}
which provides the necessary bounds for these terms.  
The second to last term in \eqref{4.17} contributes bounds of this
type in view of \eqref{4.7} and the fact that $u_0$ vanishes for large
$(t,x)$.  For $A_0$ chosen properly, the last term in \eqref{4.17}
contributes a factor which is bounded by $(A_0/20)^{1/2}(M_k(T))^{1/2}\varepsilon^{1/2}$.

What remains is to estimate the contribution to $IV$ and $V$ of the first three terms in \eqref{4.17} corresponding to $b=0$.  If we apply \eqref{4.0} to each of these with $g$ there being the first factor of each of these three terms and $h$ there being the second factor in each of these terms, we conclude that the terms corresponding to $b=0$ also enjoy the above bounds.

Clearly a similar argument applies to the corresponding terms $I$ and $II$ that only involve the Euclidean vector fields $\partial^\mu$.  These, indeed, are easier to handle since we do not need to treat the case $b=0$ separately.   We can easily adapt the above arguments to see that they satisfy the same bounds that $III$ and $IV$ enjoy.  Since this argument also implies that the remaining terms in the right side of \eqref{4.13} also satisfy these bounds, we conclude that if $\varepsilon$ is small then
$M_k(T)$ is bounded by 
\begin{equation}\label{4.19}\frac{1}{4}(M_k(T)+N_k(T))+
C\varepsilon^{3/2}
+ \frac{A_0}{2}\varepsilon\end{equation}
if $\varepsilon$ is small and $0\le T\le T_\varepsilon$, with
$T_\varepsilon$ as above with $c$ small.

We are left with estimating $N_k(T)$.  This is more straightforward and we essentially repeat the arguments from the last section to handle it.  We first notice that the arguments that were used to handle the terms $IV$ and $V$ in $M_k(T)$ will show that the last three terms in the right side of \eqref{4.14} also are bounded by the right side of \eqref{4.18}.  
All that remains is to control the second term in the right side of \eqref{4.14}, that is, the term involving the $L^1_rL^2_\omega$ norm.  To handle it, we can adapt the arguments that were used to handle the term $I$ in \eqref{36} to see that it too is bounded by the right side of \eqref{4.18}.

Since both $M_k(T)$ and $N_k(T)$ are controlled by \eqref{4.19}, we conclude that, if $\varepsilon>0$ is small enough and if $T_\varepsilon$ is as above, then  \eqref{4.15} must be valid.  Since the arguments at the end of the last section also show that under these assumptions we have that, for 
$T\le T_\varepsilon$,
$$\sum_{j\le1}\Bigl[ (\log(2+T))^{-1/2}\|\langle x\rangle^{-1/2}\nabla^j(u_k-u_{k-1})\|_{L^2(\Stk)} +\sup_{0<t<T}\|\nabla^j (u_k-u_{k-1})(t,\cd)\|_2\Bigr]$$
goes to zero geometrically as $k\to \infty$, the proof is complete.

\newsection{Some Existence Theorems in Other Dimensions}

In \cite{YY} the first and last authors showed that when $n=3$ and $\Kob$ is a star-shaped obstacle the analog of \eqref{q} has solutions with lifespan $c/\varepsilon^2$ if the 3-dimensional analog of \eqref{qdata} holds.  The proof was based on a variant of \eqref{2.15} which says that
$$\langle T\rangle^{-1/4}\sum_{|\alpha|\le N}\|\, \langle x\rangle^{-1/4} Z^\alpha w\|_{L^2(\Stk)}$$ is dominated by the right side of \eqref{2.15}.  This follows from the fact that the nonobstacle variant of this inequality holds together with the arguments from \cite{KSS} that we used here to show that the nonobstacle inequality \eqref{2.10} yields the obstacle version \eqref{2.15}.  If we combine the aforementioned variant of \eqref{2.15} with Lemmas 4.2 and 4.3, then the arguments from the last section provide a somewhat simpler proof of the existence results in \cite{YY} that avoids the use of the scaling vector field $L=t\partial_t + \langle x,\nabla_x\rangle$.

The existence results in \cite{YY} generalized to the obstacle setting one of the sharp existence theorems of Lindblad \cite{Lin} for ${\mathbb R}_+\times {\mathbb R}^3$.  We remark that the interesting problem of showing that, for the obstacle case, there is almost global existence for equations $\square u = Q(u,u',u'')$ when $Q_{uu}(0,0,0)=0$ remains open.  The nonobstacle version of this is also due to Lindblad \cite{Lin}.

We also remark that the arguments from the last section show that for
star-shaped obstacles there is always small data global existence for
equations of this type when $n\ge 5$.  In this case, no assumptions
regarding the $u$-component of the Hessian is required.  On the other
hand, the obstacle version of H\"ormander's result \cite{H} which says that when $n=4$ and $Q_{uu}(0,0,0)=0$ there is global existence for small data is open.

The methods contained herein can also be applied to \eqref{Str} when
$n=3$ and $\sigma > 2$.  In this case, one obtains global existence
for small initial data provided $\sigma > 5/2$, but the methods do not
seem to allow one to approach the critical exponent $p_c=1+\sqrt{2}$.

\end{document}